\documentclass[3p,review]{elsarticle} %authoryear(?)

\usepackage{lineno}\modulolinenumbers[5]
\journal{Computational Geometry: Theory and Applications}
\usepackage{caption}
\usepackage{subcaption}
\usepackage[plain]{fancyref}

\newcommand*{\fancyrefthmlabelprefix}{thm}
\Frefformat{plain}{\fancyrefthmlabelprefix}{%
  Theorem\fancyrefdefaultspacing#1%
}%
\frefformat{plain}{\fancyrefthmlabelprefix}{%
  theorem\fancyrefdefaultspacing#1%
}%

\newcommand*{\fancyreflemmalabelprefix}{lemma}
\Frefformat{plain}{\fancyreflemmalabelprefix}{%
  Lemma\fancyrefdefaultspacing#1%
}%
\frefformat{plain}{\fancyreflemmalabelprefix}{%
  lemma\fancyrefdefaultspacing#1%
}%

\newcommand*{\fancyrefclaimlabelprefix}{claim}
\Frefformat{plain}{\fancyrefclaimlabelprefix}{%
  Claim\fancyrefdefaultspacing#1%
}%
\frefformat{plain}{\fancyrefclaimlabelprefix}{%
  claim\fancyrefdefaultspacing#1%
}%

\newcommand*{\fancyrefdeflabelprefix}{def}
\Frefformat{plain}{\fancyrefdeflabelprefix}{%
  Definition\fancyrefdefaultspacing#1%
}%
\frefformat{plain}{\fancyrefdeflabelprefix}{%
  definition\fancyrefdefaultspacing#1%
}%

\newcommand*{\fancyrefcaselabelprefix}{case}
\Frefformat{plain}{\fancyrefcaselabelprefix}{#1}%
\frefformat{plain}{\fancyrefcaselabelprefix}{#1}%

\newcommand*{\fancyrefremlabelprefix}{rem}
\Frefformat{plain}{\fancyrefremlabelprefix}{Remark\fancyrefdefaultspacing#1}%
\frefformat{plain}{\fancyrefremlabelprefix}{remark\fancyrefdefaultspacing#1}%

\usepackage{float}
\usepackage{enumerate}
\usepackage[usenames,dvipsnames,svgnames,table]{xcolor}

\usepackage[unicode]{hyperref} %unicode,hidelinks
\hypersetup%http://en.wikibooks.org/wiki/LaTeX/Hyperlinks
{
  pdfstartview={XYZ null null 1},    % fits the width of the page to the window
  pdftitle={Partitioning orthogonal polygons into at most 8-vertex pieces, with application to an art gallery theorem},
  pdfauthor={Ervin Gy\H{o}ri, Tam\'as R\'obert Mezei},
    pdfkeywords={art gallery, polyomino, mobile guard}, % list of keywords
   colorlinks=true,        % false: boxed links; true: colored links
}

\usepackage{latexsym,amssymb,upref,amsmath,amssymb,amsfonts,amsthm}
\newtheorem{theorem}{Theorem}
\newtheorem{lemma}[theorem]{Lemma}
\newtheorem{corollary}[theorem]{Corollary}
\newtheorem{definition}[theorem]{Definition}
\newtheorem{claim}[theorem]{Claim}
\newtheorem{remark}[theorem]{Remark}

\newcommand{\ff}[1]{\left\lfloor\frac{3 #1 +4}{16}\right\rfloor}

\newcommand{\dotcup}{\mathop{\overset{\star}{\cup}}}

\newcommand{\includefigure}[1]{\includegraphics{#1}}
\newcommand{\figcap}[2]{
    \begin{figure}[H]
      \centering
      \includefigure{#1.pdf}
      \caption{#2}\label{fig:#1}
    \end{figure}}

\usepackage{tikz}
\newcommand*\circled[1]{\tikz[baseline=(char.base)]{
		\node[shape=circle,draw,inner sep=1pt] (char) {#1};}}

\begin{document}

\begin{frontmatter}
  \title{Partitioning orthogonal polygons into $\le$ 8-vertex pieces,
  with application to an art gallery theorem}

  %\tnotetext[t1]{This document is a collaborative effort.}

  \author[renyi,ceu]{Ervin Gy\H{o}ri}\fnref{fn1}\ead{ervin@renyi.hu}
  \author[ceu]{Tam\'as R\'obert Mezei\corref{cor1}}
  \ead{tamasrobert.mezei@gmail.com}

  \cortext[cor1]{Corresponding author}
  \fntext[fn1]{The research is partially supported by OTKA Grant 116 769.}
  
  \address[renyi]{MTA Alfr\'ed R\'enyi Institute of Mathematics,
          Re\'altanoda u. 13--15, 1053 Budapest, Hungary}
  \address[ceu]{Central European University,
          Department of Mathematics and its Applications,
          N\'ador u. 9, 1051 Budapest, Hungary}

  \begin{abstract}
    We prove that every simply connected orthogonal polygon of $n$
    vertices can be partitioned into $\ff{n}$ (simply connected)
    orthogonal polygons of at most 8 vertices. It yields a new and
    shorter proof of the theorem of A.~Aggarwal that $\ff{n}$
    mobile guards are sufficient to control the interior of an
    $n$-vertex orthogonal polygon. Moreover, we strengthen this
    result by requiring combinatorial guards
    (visibility is only needed at the endpoints of patrols) and
    prohibiting intersecting patrols. This yields positive answers
    to two questions of O'Rourke \citep[Section~3.4]{ORourke}.
    Our result is also a further example of the ``metatheorem'' that
    (orthogonal) art gallery theorems are based on partition
    theorems.
  \end{abstract}

  \begin{keyword}
    art gallery \sep polyomino partition \sep mobile guard
  \end{keyword}

\end{frontmatter}

%\linenumbers

\section{Introduction}

First let us define the basic object studied in this paper.

\begin{definition}\label{def:simplepolyomino}
	A simply connected orthogonal polygon or simple polyomino $P$ is the closed region of the plane
	bounded by a closed (not self-intersecting) polygon, whose
	angles are all $\pi/2$ (convex) or $3\pi/2$ (reflex). We denote the number of $P$'s
	vertices by $n(P)$.
\end{definition}

This definition implies that $n(P)$ is even. We want to emphasize the combinatorial structure of $P$ rather than its geometry, so in this paper we refer to such objects as simple polyominoes, or polyominoes for short. When we talk about not necessarily simply connected polyominoes, we explicitly state it.

\medskip

Kahn, Klawe and Kleitman \citep{K3} in 1980 proved that $\lfloor
n/4\rfloor$ guards are sometimes necessary and always sufficient to
cover the interior of a simple polyomino of $n$ vertices. Later the
first author of this paper provided a simple and short proof of

\begin{theorem}[\citep{Gy}, {\citep[Thm.~2.5]{ORourke}}]\label{thm:static}
  Every simple polyomino of $n$ vertices can be
  partitioned into $\lfloor  n/4 \rfloor$  polyominoes of at most 6
  vertices.
\end{theorem}

\Fref{thm:static} is a deeper result than that of Kahn, Klawe and
Kleitman, and gave the first hint of the existence of a
``metatheorem'': (orthogonal) art gallery theorems have underlying
partition theorems. The general case was proved by Hoffmann
\citep{Hof}.

\begin{theorem}[\citep{Hof}]\label{thm:Hoffmann}
  Every (not necessarily simply connected!) orthogonal polygon with
  $n$ vertices can be covered by $\lfloor n/4\rfloor$ guards.
\end{theorem}

Hoffmann's method  (partitioning into smaller pieces that can be
covered by one guard) is another example of the metatheorem.

\medskip

In this paper, we present further evidence that the metatheorem holds,
namely we prove the following partition theorem:

\begin{theorem}\label{thm:mobile}
  Any simple polyomino of $n$ vertices
  can be partitioned into at most $\ff{n}$ simple polyominoes of
  at most 8 vertices.
\end{theorem}

A mobile guard is one who can patrol a line segment, and it covers a point $x$ of the gallery if there is a point $y$ on its patrol such that the line segment $[x,y]$ is contained in the gallery. The mobile guard art gallery theorem for simple polyominoes follows immediately from \Fref{thm:mobile}, as a polyomino of at most 8 vertices can be covered by a mobile guard:

\begin{theorem}[\citep{Ag}, proof also in {\citep[Thm.~3.3]{ORourke}}]
\label{thm:mobilecover}
  $\ff{n}$ mobile guards are sufficient for covering an $n$ vertex
  simple polyomino.
\end{theorem}

\Fref{thm:mobile} is a stronger result than \Fref{thm:mobilecover} and it is
interesting on its own. It fits into the series of  results in
\citep{Gy}, \citep{Hof},  \citep[Thm.~2.5]{ORourke}, \citep{GyH} showing that
rectilinear art gallery theorems are based on theorems on partitions of
polyominoes into smaller (``one guardable'') polyominoes.

\medskip

Moreover, \Fref{thm:mobile} directly implies the following corollary
which strengthens the previous theorem and answers two
questions raised by O'Rourke \citep[Section 3.4]{ORourke}.

\begin{corollary}\label{thm:mobilecoverstrong}
  $\ff{n}$ mobile guards are sufficient for covering an $n$ vertex
  simple polyomino such that the patrols of two guards do
  not pass through one another and visibility is only required at the
  endpoints of the patrols.
\end{corollary}

The proof of \Fref{thm:mobile} is similar to the
proofs of \Fref{thm:static} in that it finds a suitable cut and then
uses induction on the parts created by the cut. However, here a cut
along a line segment connecting two reflex vertices is not automatically good.
In case we have no such cuts, we also rely heavily on a tree structure
of the polyomino (\Fref{sec:treestructure}). However, while O'Rourke's
proof only uses straight cuts, in our case this is not sufficient:
\Fref{fig:LisNeccessary} shows a polyomino of 14 vertices which cannot
be cut into 2 polyominoes of at most 8 vertices using cuts along
straight lines. Therefore we must consider L-shaped cuts too.
\figcap{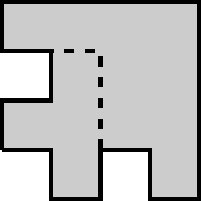}{An L-shaped cut creating a partition into 8-vertex polyominoes}

Interested readers can find a thorough introduction to the subject of art gallery
problems in \citep{ORourke}.

\section{Definitions and preliminaries}

Let $P,P_1,P_2$ be simple polyominoes of $n,n_1,n_2$ vertices, respectively. If $P=P_1\cup P_2$, $\mathrm{int}(P_1)\cap \mathrm{int}(P_2)=\emptyset$, $0<n_1,n_2$ and $n_1+n_2\le n+2$ are satisfied, we say that $P_1,P_2$ form an admissible partition of $P$, which we denote by $P=P_1\dotcup P_2$. Also, we call $L=P_1\cap P_2$ a cut in this case. We may describe this relationship concisely by $L(P_1,P_2)$. If, say, we have a number of cuts $L_1,L_2,$ etc.,~then we usually write $L_i(P_1^i,P_2^i)$. Generally, if a subpolyomino is denoted by $P_x^y$, then $y$ refers to a cut and $x\in\{1,2\}$ is the label of the piece in the partition created by said cut.
%, and let a directed cut be a cut plus an order on the parts, which we denote by $\vr L=(P_1,P_2)$.
Furthermore, if
	\begin{equation}\label{eq:n1n2}
    \ff{n_1}+\ff{n_2}\le \ff{n}.
  \end{equation}
is also satisfied, we say that $P_1,P_2$ form an induction-good partition of $P$, and we call $L$ a \textbf{good~cut}.
  
\begin{lemma}\label{lemma:tech}
 	An admissible partition $P=P_1\dotcup P_2$ is also and induction-good partition if \\
 	\textbf{\mbox{}\quad (a)}\makeatletter\edef\@currentlabel{\arabic{theorem}(a)}\makeatother\label{lemma:tech:a} $n_1+n_2= n+2$ and $n_1\equiv 2,8,\text{ or }14 \pmod{16}$, \\
 	\mbox{}\qquad\qquad\textbf{or}\\
 	\textbf{\mbox{}\quad (b)}\makeatletter\edef\@currentlabel{\arabic{theorem}(b)}\makeatother\label{lemma:tech:b} $n_1+n_2= n$ and $n_1\equiv 0,2,6,8,12,\text{ or }14 \pmod{16}$, \\
 	\mbox{}\qquad\qquad\textbf{or}\\
 	\textbf{\mbox{}\quad (c)}\makeatletter\edef\@currentlabel{\arabic{theorem}(c)}\makeatother\label{lemma:tech:c} $n\not\equiv 14\mod 16$ and either
 	$$n_1+n_2=n\text{ and }n_1\equiv 10 \pmod{16}\textbf{\quad or \quad}
 	n_1+n_2=n+2\text{ and }n_1\equiv 12 \pmod{16}.$$
\end{lemma}
\begin{proof}
	Using the fact that the floor function satisfies the triangle inequality, the proof reduces to an easy case-by-case analysis, which we leave to the reader.
 	%  In case \textbf{a)}, an equivalent formulation of \Fref{eq:n1n2} is
 	%    \begin{equation*}
 	%      \fff{n_1}+\fff{n_2}\ge \fff{n}+\frac{10}{16}
 	%    \end{equation*}
 	%  As the fractional part function satisfies the triangle inequality,
 	%  the previous inequality holds if
 	%    \begin{equation*}
 	%      \fff{n_1}\ge
 	%      \mathrm{frac}\left(\frac{3n_1-6}{16}\right)+\frac{10}{16}
 	%    \end{equation*}
 	%  The last inequality holds when $n(P_1)\equiv 2,8,14 \mod
 	%  16$. The proof of case \textbf{b)} and \textbf{c)} goes similarly. 
\end{proof}

Any cut $L$ in this paper falls into one of the following 3 categories (see \Fref{fig:cuts}):
\begin{enumerate}[\bf (a)]
  \item \textbf{1-cuts:} $L$ is a line segment, and exactly one of its
  endpoints is a (reflex) vertex of $P$.

  \item \textbf{2-cuts:} $L$ is a line segment, and both of its
  endpoints are (reflex) vertices of $P$.

  \item \textbf{L-cuts:} $L$ consists of two connected line
  segments, and both endpoints of $L$ are (reflex) vertices of $P$.
\end{enumerate}

Note that for 1-cuts and L-cuts the size of the parts satisfy $n_1+n_2=n+2$, while for 2-cuts we have $n_1+n_2=n$.
%Given a cut $L$ of $P$, we denote the two parts of the partition generated by $L$ with $P_1(L)$ and $P_2(L)$, where $P_1(L)$ is the part which is downwards or if $L$ is vertical to the right of $L$.

%\begin{claim}\label{claim:glue}
%	Suppose $P_1$ and $P_2$ are two polyominoes, and $L:=P_1\cap P_2$
%	is a polygonal path. Then $P=P_1\dotcup P_2$ is a
%	polyomino-partition which is also defined by $L$ as in
%	\Fref{lemma:polygoncut}.
%\end{claim}
%
%Now we may refer to two part admissible partitions as cuts and
%vice-versa.

\begin{figure}[H]
	\centering
	\begin{subfigure}{.50\textwidth}
		\renewcommand{\thesubfigure}{a}
		\centering
		\includefigure{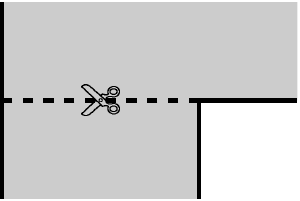}
		\caption{1-cut}
	\end{subfigure}%
	\begin{subfigure}{.50\textwidth}
		\renewcommand{\thesubfigure}{b1}
		\centering
		\includefigure{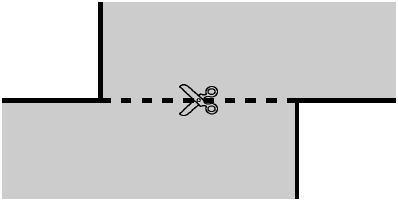}
		\caption{2-cut}
	\end{subfigure}
	
	\bigskip
	
	\begin{subfigure}{.50\textwidth}
		\renewcommand{\thesubfigure}{b2}
		\centering
		\includefigure{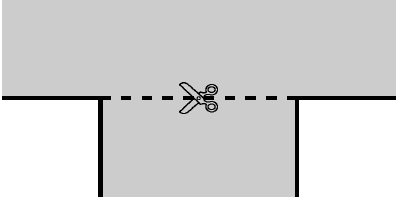}
		\caption{2-cut}
	\end{subfigure}%
	\begin{subfigure}{.50\textwidth}
		\renewcommand{\thesubfigure}{c}
		\centering
		\includefigure{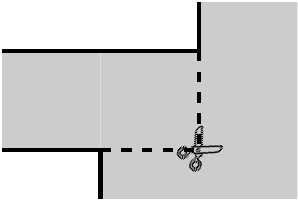}
		\caption{L-cut}
	\end{subfigure}
	\caption{Examples for all types of cuts. Light gray areas are subsets of $\mathrm{int}(P)$.}\label{fig:cuts}
\end{figure}

In the proof of \Fref{thm:mobile} we are searching for an induction-good partition of $P$. As a good cut defines an induction-good partition, it is sufficient to find a good cut. We could hope that a good cut of a subpolyomino of $P$ is extendable to a good cut of $P$, but unfortunately a good cut of a subpolyomino may only be an admissible cut with respect to $P$ (if it is a cut of $P$ at all). \Fref{lemma:tech}, however, allows us to look for cut-systems containing a good cut. % and are preserved under broader conditions.
Fortunately, it is sufficient to consider non-crossing, nested cut-systems of at most 3 cuts, defined as follows.

\begin{definition}[Good cut-system]\label{def:goodcutsystem} 
The cuts $L_1(P_1^1,P_2^1)$, $L_2(P_1^2,P_2^2)$ and $L_3(P_1^3,P_2^3)$ (possibly $L_2=L_3$) constitute a good cut-system if $P_1^1\subset P_1^2\subseteq P_1^3$, and the set 
	$$\left\{n(P_1^i)\ |\ i\in\{1,2,3\}\right\}\cup \left\{n(P_1^i)+2\ |\ i\in\{1,2,3\}\text{ and }L_i\text{ is a 2-cut}\right\}$$
	contains three consecutive even elements modulo 16 (ie., the union of their residue classes contains a subset of the form $\{a,a+2,a+4\}+16\mathbb{Z}$). If this is the case we also define their kernel as $\ker\{L_1,L_2,L_3\}=(P_1^1-L_1)\cup (P_2^3-L_3)$, which will be used in \Fref{lemma:extendconsecutives}.
%  Let $\mathcal{L}$ be a set of cuts of $P$, where $|\mathcal{L}|=2$ or $|\mathcal{L}|=3$. Let $I=\{1\le i\le |\mathcal{L}|\}\cap \mathbb{N}$ We say that $\mathcal{L}$ is a good cut-system (of $P$) if $\exists \tau\in \{1,2\}^I$ and $\exists\pi\in \mathrm{Sym}(I)$ (a permutation of $I$) such that $P_{\tau(1)}(L_{\pi(1)})\subseteq P_{\tau(2)}(L_{\pi(2)}) \Big(\subseteq P_{\tau(3)}(L_{\pi(3)})\Big)$, and the set
%    \begin{align*}
%      W_{\tau,\pi}(\mathcal{L},P)=\Big\{ n(P_{\tau(i)}(L_{\pi(i)}))\ |\ i\in I\Big\}\cup \Big\{n(P_{\tau(i)}(L_{\pi(1)}))+2\ |\ i\in I\text{ and $L_{\pi(i)}$ is a 2-cut}\Big\}
%    \end{align*}
%   contains three consecutive even elements modulo $16$. Lastly, define $$\ker\mathcal{L}:=P_{\tau(1)}(L_{\pi(1)})\cup P_{3-\tau(3)}(L_{\pi(3)})-L_{\pi(1)}-L_{\pi(3)}.$$
\end{definition}

\Fref{lemma:tech:a} and \ref{lemma:tech:b} immediately yield that any good cut-system contains a good cut.

\begin{remark}\label{rem:invert}
	It is easy to see that if a set of cuts satisfies this definition, then they obviously satisfy it in the reverse order too (the order of the generated parts are also switched).
	%The only non-trivial part in this statement is that $x\mapsto n(P)+2-x$ is a bijection between the sets that contain the three consecutive even elements modulo 16.
%\end{remark}
%\begin{remark}\label{rem:cutset}
	Actually, these are exactly the two orders in which they do so. Thus the kernel is well-defined, and when speaking about a good cut-system it is often enough to specify the set of participating cuts.
\end{remark}

%\begin{claim}\label{claim:goodsystem}
%	Any good cut-system contains at least one good cut.
%\end{claim}
%\begin{proof}
%	Follows immediately from \Fref{lemma:tech}.
%\end{proof}

%The reason for the introduction of the previous definition is that \Fref{lemma:tech} implies that any good cut-system contains a good cut.

\medskip

%\begin{proof}
%%	This follows immediately from \Fref{lemma:tech}.
%	If the modulo 16 remainder of $n(P_1^{L_i,s})$ is one of $2,8,14$ for
%	some $i\in I$, then $L_i$ is a good cut by \Fref{lemma:tech}.
%	Otherwise there exists $i\in I$ s.t. $n(P_1^{L_i,s})+2$ is congruent to one of
%	$2,8,14$ modulo $16$, therefore $n(P_1^{L_i,s})$ is congruent to one of $0,6,12$ modulo $16$, thus $L_i$ is a good 2-cut by	\Fref{lemma:tech:b}.
%\end{proof}

%Furthermore the following lemma, which we can use to change our point of view from $s$ to another point of $P$, follows immediately.
%\begin{claim}\label{claim:invert}
%	Let $\{ L_i\ |\ i\in I\}$ be a good cut-system with respect to $s$.
%	If $s'\in \cap_{i\in I}\mathrm{int}(P_j^{L_i, s})$ for some $j\in\{1,2\}$, then $\{	L_i\ |\ i\in I\}$ is a good cut-system with respect to $s'$ too.
%\end{claim}

\section{Tree structure}\label{sec:treestructure}

% ez a claim előtti részhez lehet, hogy kéne valami bizonyítás még,
% esetleg úgy definálni az éleket, hogy a csúcsban találkozást is
% tudjuk kezelni (hogy ne csak simple polyominokra menjen)

Any reflex vertex of a polyomino defines a (1- or 2-) cut along a horizontal line segment whose interior is contained in $\mathrm{int}(P)$ and whose
endpoints are the reflex vertex and another point on the boundary of
the polyomino. Next we define a graph structure derived from $P$, which is a standard tool in the literature, for example it is called the $R$-graph of $P$ in \citep{GyH}. % and it is a standard tool in the literature, see for example \citep[p.~76]{ORourke}.

\medskip

\begin{definition}[Horizontal cut tree]\label{def:tree}
  The horizontal cut tree $T$ is obtained as follows. First, partition $P$ into a set of rectangles by cutting along all of the horizontal cuts of $P$. Let $V(T)$, the vertex set of $T$ be the set of resulting (internally disjoint) rectangles. Two rectangles of $T$ are connected by an edge in $E(T)$ iff their boundaries intersect.
\end{definition}

\medskip

The graph $T$ is indeed a tree as its connectedness is trivial and since any cut creates two internally disjoint polyominoes, $T$ is also cycle-free. We can think of $T$ as a sort of dual of the planar graph determined by $P$ and its horizontal cuts. The nodes of $T$ represent rectangles of $P$ and edges of $T$ represent horizontal 1- and 2-cuts of $P$. For this reason we may refer to nodes of $T$ as rectangles. This nomenclature also helps in distinguishing between vertices of $P$ (points) and nodes of $T$.  Moreover, for an edge $e\in E(T)$, we may denote the cut represented by $e$ by simply $e$, as the context should make it clear whether we are working in the graph $T$ or the polyomino $P$ itself.

\medskip

Note that the vertical sides of rectangles are also edges of the polyomino. We will also briefly use vertical cut trees, which can be defined analogously.

%\begin{claim}
%  
%\end{claim}
%\begin{proof}
%   Given any two rectangles in $V(T)$, there exists a polygonal path
%   $M$ in $\mathrm{int}(P)$ which connects their center points.
%   A corresponding walk in $T$ can be obtained by taking the
%   rectangles and horizontal cuts $M$ successively intersects.
%   Therefore $T$ is connected. Also, cutting at any horizontal cut
%   (edge of $T$) makes $P$ disconnected, so $T$ is cycle-free.
%\end{proof}

\begin{definition}\label{def:tfunc}
  Let $T$ be the horizontal cut tree of $P$. Define $t:E(T)\to \mathbb{N}$ as follows:
  given any edge $\{R_1,R_2\}\in E(T)$, let
  $$t(\{R_1,R_2\})=n(R_1\cup R_2)-8.$$
\end{definition}

Observe that
$$t(e)=\left\{
\begin{array}{rl}
 0, &\text{ if $e$ represents a 2-cut;} \\
-2, &\text{ if $e$ represents a 1-cut.} \\
\end{array}
\right.$$
The following claim is used throughout the paper to count the number of vertices of subpolyominoes.
\begin{claim}\label{claim:tsize} Let $T$ be the horizontal cut tree of $P$. Then
  $$n(P)=4|V(T)|+\sum_{e\in E(T)} t(e).$$
\end{claim}
\begin{proof}
	The proof is straightforward. 
%	Traverse the simple polygon that is the boundary of $P$. Write $-2$ on points of $\partial P$ which are vertices of one of the rectangles, but not vertices of $\partial P$.
%
%	By induction on $|V(T)|$. Let $R_1$ be a leaf in $T$, and let $R_2$ be its neighbor. Then $$P'=\bigcup_{R\in V(T)-R_1} R$$ is a polyomino and $V(T)-R_1$ is one of its horizontal cut trees. Applying the induction hypothesis we have
%\begin{align*}
%	n(P')&=4(|V(T)|-1)+\sum_{e\in E(T)-\{R_1,R_2\}}t(e)=n(P)-(4+t(e))=\\
%	&=n(P)-(n(R_1\cup R_2)-4)=n(P)-(n(R_1\cup R_2)-n(R_2))
%\end{align*}
%To finish the proof, just verify that $n(P)=n(P')-n(R_2)+n(R_1\cup R_2)$.
\end{proof}
\begin{remark}\label{rem:refine}
	The equality still holds even if some of the rectangles of $T$ are cut into several rows (and the corresponding edges, for which the function $t$ takes $-4$, are added to $T$). % of equal width
\end{remark}

\subsection{Extending cuts and cut-systems}
The following two technical lemmas considerably simplify our analysis in \Fref{sec:proof}, where many cases distinguished by the relative positions of reflex vertices of $P$ on the boundary of a rectangle need to be handled. For a rectangle $R$ let us denote its top left, top right, bottom left, and bottom right vertices with $v_{TL}(R)$, $v_{TR}(R)$, $v_{BL}(R)$, and $v_{BR}(R)$, respectively.

\begin{figure}[H]
	\centering
	\begin{subfigure}{.33\textwidth}
		\centering
		\includefigure{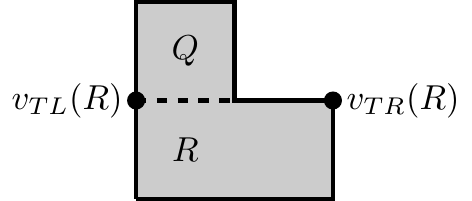}
		\caption{$Q\subseteq R_{TL}$, and $e_{TL}(R)$ is a 1-cut}
		\label{fig:RTL:a}
	\end{subfigure}%
	\begin{subfigure}{.34\textwidth}
		\centering
		\includefigure{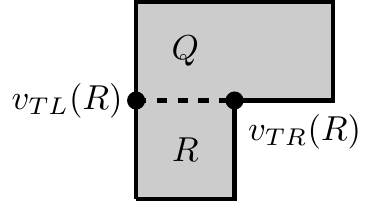}
		\caption{$Q\subseteq R_{TL}$, $e_{TL}(R)$ is a 1-cut, and $R_{TR}=\emptyset$}
		\label{fig:RTL:b}
	\end{subfigure}%
	\begin{subfigure}{.33\textwidth}
		\centering
		\includefigure{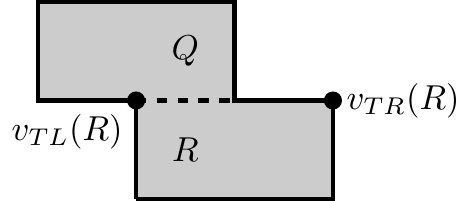}
		\caption{$Q\subseteq R_{TL}$, and $e_{TL}(R)$ is a 2-cut}
		\label{fig:RTL:c}
	\end{subfigure}%
	
	\bigskip
	
	\begin{subfigure}{.33\textwidth}
		\centering
		\includefigure{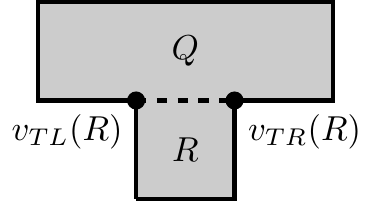}
		\caption{$R_{TL}=R_{TR}=\emptyset$, $R$ is either a corridor or a pocket (see \Fref{sec:proof})}
		\label{fig:RTL:d}
	\end{subfigure}%
	\begin{subfigure}{.33\textwidth}
		\centering
		\includefigure{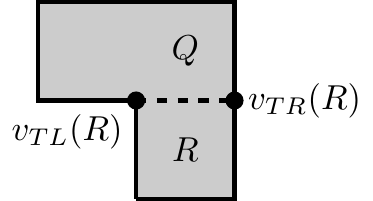}
		\caption{$R_{TL}=\emptyset$, but $Q\subseteq R_{TR}$}
		\label{fig:RTL:e}
	\end{subfigure}
		\caption{$R\cup Q$ in all essentially different relative positions of $R,Q\in V(T)$, where $\{R,Q\}\in E(T)$ and $v_{TL}(R)\in Q$}\label{fig:RTL}
\end{figure}

\begin{definition}\label{rectvertexns}
  Let $R,Q\in V(T)$ be arbitrary. We say that $Q$ is adjacent to $R$
  at $v_{TL}(R)$, if $v_{TL}(R)\in Q$ and $v_{TL}(R)$ is not a vertex of the polyomino $R\cup Q$, or $v_{TR}(R)\notin Q$. Such situations are depicted on Figures \ref{fig:RTL:a}, \ref{fig:RTL:b}, and \ref{fig:RTL:c}. However, in the case of Figure \ref{fig:RTL:d} and \ref{fig:RTL:e} we have $v_{TR}(R)\in Q\not\subseteq R_{TL}\;(=\emptyset)$.
  
  \medskip
  
  If $Q$ is adjacent to $R$ at $v_{TL}(R)$, let $e_{TL}(R)=\{R,Q\}$; by cutting $P$ along the dual of $e_{TL}(R)$, ie., $R\cap Q$, we get two polyominoes, and we denote the part containing $Q$ by $R_{TL}$. If there is no such $Q$, let $e_{TL}(R)=\emptyset$ and $R_{TL}=\emptyset$. These relations can be defined analogously for top right ($R_{TR}$, $e_{TR}(R)$), bottom left ($R_{BL}$, $e_{BL}(R)$), and bottom right ($R_{BR}$, $e_{BR}(R)$).
\end{definition}

\begin{lemma}\label{lemma:extend}
  Let $R$ be an arbitrary rectangle such that $R_{BL}\neq\emptyset$. Let
  $U$ be the remaining portion of the polyomino, ie., $P=R_{BL}\dotcup U$ is
  a polyomino-partition. Take an admissible polyomino-partition $U=U_1\dotcup U_2$ such that $v_{BL}(R)\in U_1$.  We can extend this to an admissible polyomino-partition of $P$ where the two parts are $U_1\cup R_{BL}$ and $U_2$.
\end{lemma}
\begin{proof}
  Let $Q_1=R\cap U_1$ and let $Q_2\in V(T)$ be the rectangle which is a subset of $R_{BL}$ and adjacent to $R$.

  \medskip

  Observe that $U_1$ and $R_{BL}$ only intersect on $R$'s bottom
  side, therefore their intersection is a line segment $L$ and so
  $U_1\cup R_{BL}$ is a polyomino. Trivially, $P=(U_1\cup R_{BL})\dotcup U_2$ is partition into polyominoes, so only  admissibility remains to be checked.
  
  \medskip
  
  Let the horizontal cut tree of $U_1$ and $R_{BL}$ be $T_{U_1}$ and $T_{R_{BL}}$, respectively.  The horizontal cut tree of $U_1\cup R_{BL}$ is $T_{U_1}\cup T_{R_{BL}}+\{Q_1,Q_2\}$, except if $t(\{Q_1,Q_2\})=-4$. Either way, by referring to \Fref{rem:refine} we can use \Fref{claim:tsize} to write that
    \begin{align}\label{eq:newsize}
      n(U_1\cup R_{BL})&+n(U_2)-n(P)=n(U_1)+n(R_{BL})+t(\{Q_1,Q_2\})+n(U_2)-n(P)=\nonumber\\
      &=\Big(n(U_1)+n(U_2)-n(U)\Big)+\Big(n(U)+n(R_{BL})-n(P)\Big)+t(\{Q_1,Q_2\})=\\
      &=(n(U_1)+n(U_2)-n(U))-t(\{R,Q_2\})+t(\{Q_1,Q_2\}).\nonumber
    \end{align}
  Now it is enough to prove that $t(\{Q_1,Q_2\})\le t(\{R,Q_2\})$. If $t(\{R,Q_2\})=0$ this is trivial. The remaining case is when $t(\{R,Q_2\})=-2$. This means that $v_{BL}(R)$ is not a vertex of $R\cup Q_2$, therefore it is not a vertex of $Q_1\cup Q_2$ either, implying that $n(Q_1\cup Q_2)<8$.
\end{proof}

\begin{lemma}\label{lemma:extendconsecutives}
  Let $R\in V(T)$ be such that $R_{BL}\neq\emptyset$. Let $U$ be the other half of the polyomino, ie., $P=R_{BL}\dotcup U$. If $U$ has a good cut-system $\mathcal{L}$ such that $v_{BL}(R)\in\ker\mathcal{L}$ then $P$ also has a good cut-system.
\end{lemma}
\begin{proof}
  Let us enumerate the elements of $\mathcal{L}$ as $L_i$ where $i\in I$. Take $L_i(U_1^i,U_2^i)$ such that $v_{BL}(R)\in U_1^i$. Using \Fref{lemma:extend} extend $L_i$ to a cut $L_i'(P_1^i,P_2^i)$ of $P$ such that $U_2^i=P_2^i$. 
  
  \medskip
  
  \Fref{eq:newsize} and the statement following it implies that $n(P_1^i)+n(P_2^i)=n(P)+2\implies n(U_1^i)+n(U_2^i)=n(U)+2$. In other words, if $L_i$ is a 2-cut then so is $L_i'$. Therefore
  \begin{align*}
	  \left\{n(U_2^i)\ |\ i\in I\right\}&\cup \left\{n(U_2^i)+2\ |\ i\in I\text{ and }L_i\text{ is a 2-cut}\right\}\subseteq \\
	  &\subseteq \left\{n(P_2^i)\ |\ i\in I\right\}\cup \left\{n(P_2^i)+2\ |\ i\in I\text{ and }L_i'\text{ is a 2-cut}\right\}
  \end{align*}
  and by referring to \Fref{rem:invert} we get that $\{L_i'\ |\ i\in I\}$ is a good cut-system of $P$.
\end{proof}

\section{Proof of \texorpdfstring{\Fref{thm:mobile}}{Theorem~3}}\label{sec:proof}

% Case magic
\newcounter{case}
\setcounter{case}{0}
\renewcommand{\thecase}{Case~\arabic{case}}% \arabic{section}
\newcommand{\casemark}[1]{}
\makeatletter
\newcommand\case{\@startsection{case}{2}{\z@}%
	{12\p@ \@plus 6\p@ \@minus 3\p@}%
	{3\p@ \@plus 6\p@ \@minus 3\p@}%
	{\normalfont\normalsize\itshape\bfseries\boldmath}}

\newcounter{subcase}[case]
\setcounter{subcase}{0}
\renewcommand{\thesubcase}{Case~\arabic{case}.\arabic{subcase}}%
\newcommand{\subcasemark}[1]{}
\newcommand\subcase{\@startsection{subcase}{3}{\z@}%
	{12\p@ \@plus 6\p@ \@minus 3\p@}%
	{\p@}%
	{\normalfont\normalsize\itshape\bfseries\boldmath}}

\newcounter{subsubcase}[subcase]
\setcounter{subsubcase}{0}
\renewcommand{\thesubsubcase}{Case~\arabic{case}.\arabic{subcase}.\arabic{subsubcase}}%
\newcommand{\subsubcasemark}[1]{}
\newcommand\subsubcase{\@startsection{subsubcase}{3}{\z@}%
	{12\p@ \@plus 6\p@ \@minus 3\p@}%
	{\p@}%
	{\normalfont\normalsize\itshape\bfseries\boldmath}}

\newcommand{\toclevel@case}{2} 
\newcommand{\toclevel@subcase}{3} 
\newcommand{\toclevel@subsubcase}{4} 
\makeatother

We will prove \Fref{thm:mobile} by induction on the number of
vertices. For $n\le 8$ the theorem is trivial.

\medskip

For $n>8$, let $P$ be the polyomino which we want to partition into
smaller polyominoes. It is enough to prove that $P$ has a good cut. The rest of this proof is an extensive case study. Let $T$ be the horizontal cut tree of $P$. We need two more definitions.

\begin{itemize}
	\item A \textbf{pocket} in $T$ is a degree-1 rectangle $R$, whose only incident edge in $T$ is a 2-cut of $P$ and this cut covers the entire top or bottom side of $R$.
	\item A \textbf{corridor} in $T$ is a rectangle $R$ of degree $\ge 2$ in $T$, which has an incident edge in $T$ which is a 2-cut of $P$ and this cut covers the entire top or bottom side of $R$.
\end{itemize}

We distinguish 4 cases.
\begin{description}
	\itemsep0em
	\setlength{\baselineskip}{14pt}
	\item[\quad\Fref{case:Tisapath}.] $T$ is a path, \Fref{fig:cases}(a);
	\item[\quad\Fref{case:corridors}.] $T$ has a corridor, \Fref{fig:cases}(b);
	\item[\quad\Fref{case:pockets}.] $T$ does not have a corridor, but it has a pocket, \Fref{fig:cases}(c); 
	\item[\quad\Fref{case:twisted}.] None of the previous cases apply, \Fref{fig:cases}(d).
\end{description}

\begin{figure}[H]
	\centering
	\begin{subfigure}{.34\textwidth}
		\centering
		\includefigure{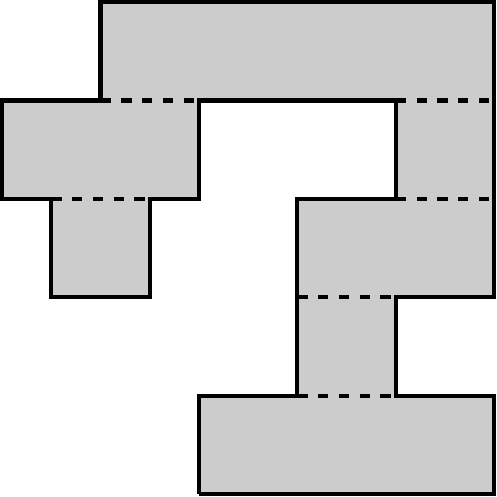}
		\caption{$T$ is a path.}
	\end{subfigure}%
	\begin{subfigure}{.66\textwidth}
		\centering
		\includefigure{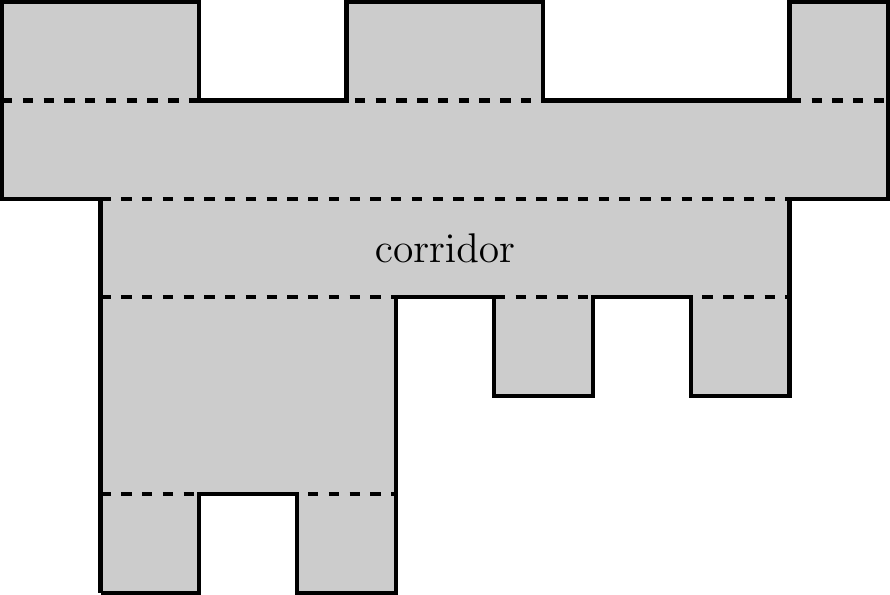}
		\caption{$T$ has a corridor.}
	\end{subfigure}
	
	\bigskip
	
	\begin{subfigure}{.34\textwidth}
		\centering
		\includefigure{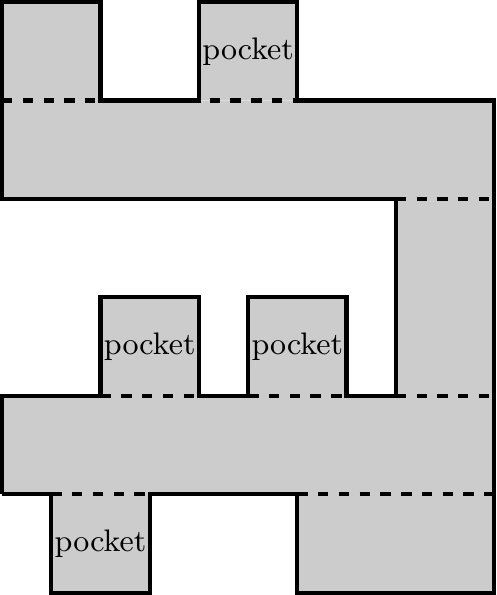}
		\caption{$T$ does not have a corridor, but it has a pocket.}
	\end{subfigure}%
	\begin{subfigure}{.64\textwidth}
		\centering
		\includefigure{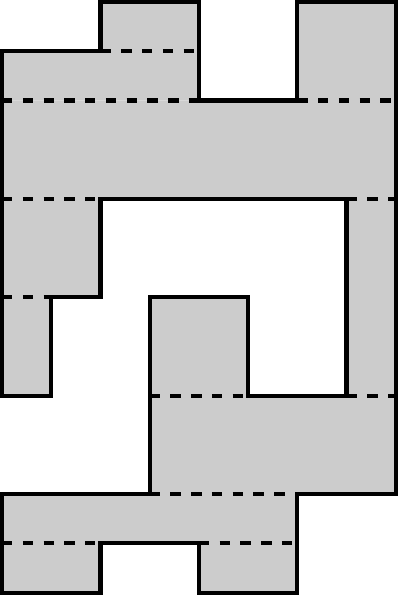}
		\caption{$T$ does not have a corridor or a pocket, and it is not a path.}
	\end{subfigure}
	\caption{The 4 cases of the proof.}\label{fig:cases}
\end{figure}

\case{\texorpdfstring{$T$}{T} is a path}\label{case:Tisapath}

\begin{claim}\label{claim:deg2doublecut}
	If an edge incident to a degree-2 vertex $R$ of $T$ is a
	2-cut of $P$ then the incident edges of $R$ form a good cut-system.
\end{claim}
\begin{proof}
	Let the two incident edges of $R$ be $e_1$ and $e_2$. Let their generated partitions be $e_1(P_1^1,P_2^1)$ and $e_2(P_1^2,P_2^2)$, such that $R\subseteq P_2^1\cap P_1^2$. Then $P_1^2=P_1^1\cup R$, so
	\begin{align*}
		n(P_1^2)=n(P_1^1)+n(R)+t(e_1)=n(P_1^1)+4.
	\end{align*}
	\Fref{def:goodcutsystem} is satisfied by $\{e_1,e_2\}$, as $\{n(P_1^1),n(P_1^2)\}\cup \{n(P_1^1)+2\}$ is a set of three consecutive even elements.
\end{proof}

\begin{claim}\label{claim:deg2path}
	If there are two rectangles $R_1$ and $R_2$ which are adjacent degree-2 vertices of $T$ then the union of the set of incident edges of $R_1$ and $R_2$ form a good cut-system.
\end{claim}
\begin{proof}
	Let the two components of $T-R_1-R_2$ be $T_1$ and $T_2$, so that
	$e_1,e_2,f\in E(T)$ joins $T_1$ and $R_1$, $R_1$ and $R_2$, $R_2$ and $T_2$, respectively.
%	$$T_1\stackrel{e_1}{\longleftrightarrow} R_1
%	\stackrel{f}{\longleftrightarrow} R_2
%	\stackrel{e_2}{\longleftrightarrow} T_2.$$
	Obviously, $\cup V(T_1)\subset (\cup V(T_1))\cup R_1 \subset (\cup V(T_1))\cup R_1\cup R_2$. If one of $\{e_1,e_2,f\}$ is a 2-cut, we are done by the previous claim. Otherwise
	\begin{align*}
	n((\cup V(T_1))\cup R_1)=n(\cup V(T_1))+n(R_1)+t(e_1)&=n(\cup V(T_1))+2, \\
	n((\cup V(T_1))\cup R_1\cup R_2)=n(\cup V(T_1))+n(R_1)+n(R_2)+t(e_1)+t(f)&=n(\cup V(T_1))+4,
	\end{align*}
	and so $\{n(\cup V(T_1)),n((\cup V(T_1))\cup R_1),n((\cup V(T_1))\cup R_1\cup R_2)\}$ are three consecutive even elements. This concludes the proof that $\{e_1,e_2,f\}$ is a good cut-system of $P$.
\end{proof}

Suppose $T$ is a path. If $T$ is a path of length $\le 3$ such that all of its edges are 1-cuts, then $n(P)\le 8$. Also, if $T$ is path of length 2 and its only edge represents a 2-cut, then $n(P)=8$. Otherwise either \Fref{claim:deg2doublecut} or \Fref{claim:deg2path} can be applied to provide a good cut-system of $P$.

\case{\texorpdfstring{$T$}{T} has a corridor}\label{case:corridors}

	Let	$e=\{R',R\}\in E(T)$ be a horizontal 2-cut such that $R'$
	is a wider rectangle than $R$ and $\deg(R)\ge 2$. Let the generated partition be $e(P_1^e,P_2^e)$ such that $R'\subseteq P_1^e$. We can handle all possible cases as follows.
	\begin{enumerate}[(a)]

		\item If $n(P_1^e)\not\equiv 4,10\bmod 16$ or $n(P_2^e)\not\equiv
		4,10\bmod 16$, then $e$ is a good cut by \Fref{lemma:tech:b}.

		\item If $\deg(R)=2$, we find a good cut using
		\Fref{claim:deg2doublecut}.

		\item If $R_{BL}=\emptyset$, then $L(P_1^L,P_2^L)$ such that $R'\subseteq P_1^L$ in \Fref{fig:corridor_b} is a good cut, since $n(P_1^L)=n(P_1^e)+4-0\equiv 8,14 \bmod 16$.
    \figcap{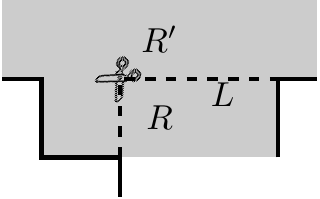}{$L$ is a good cut}

    \item If $R_{BL}\neq\emptyset$ and $\deg(R)\ge 3$, then let us consider the following five cuts of $P$ (\Fref{fig:corridor_c}): $L_1(R_{BL},R\cup P_1^e)$, $L_2(R_{BL}\cup Q_1, Q_2\cup Q_3\cup P_1^e)$, $L_3(R_{BL}\cup Q_1\cup Q_2,Q_3\cup P_1^e)$, $L_4(Q_3, R_{BL}\cup Q_1\cup Q_2\cup P_1^e)$, and $L_5(Q_3\cup Q_2, R_{BL}\cup Q_1\cup P_1^e)$.
    \figcap{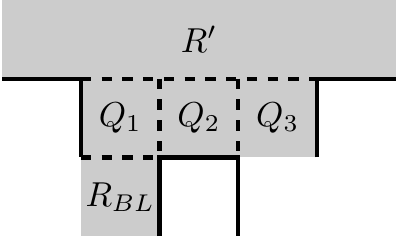}{$\deg(R)\ge 3$ and $R_{BL}\neq\emptyset$}
    The first piece of these partitions have the following number of vertices (respectively).
    \begin{enumerate}[\bfseries (1)]
      \setlength{\baselineskip}{16pt}
      \item $n(R_{BL})$
      \item $n(R_{BL}\cup Q_1)=n(R_{BL})+n(Q_1)+(t(e_{BL}(R))-2)=n(R_{BL})+t(e_{BL}(R))+2$
      \item $n(R_{BL}\cup Q_1\cup Q_2)=n(R_{BL})+n(Q_1\cup Q_2)+t(e_{BL}(R))=n(R_{BL})+t(e_{BL}(R))+4$
      \item $n(Q_3)$
      \item $n(Q_3\cup Q_2)=n(Q_3)+n(Q_2)-2=n(Q_3)+2$
    \end{enumerate}
    %Suppose none of the 5 cuts above are good cuts.
    \begin{itemize}
      \item If $t(e_{BL}(R))=0$, then $\{L_1,L_2,L_3\}$ is a good cut-system, so one
      of them is a good cut.

      \item If $t(e_{BL}(R))=-2$, and none of the 5 cuts above are good cuts, then using \Fref{lemma:tech:b} on
      $L_2$ and $L_3$ gives ${n(R_{BL})\equiv 4,10\bmod 16}$.
      The same argument can be used on $L_4$ and $L_5$ to conclude that
      $n(Q_3)\equiv 4,10\bmod 16$.
      However, previously we derived that
      \begin{align*}
        n(P_2^e)&\equiv 4,10\bmod 16 \\
        n(R_{BL}\cup Q_1\cup Q_2\cup Q_3)=n(R_{BL}\cup Q_1\cup Q_2)+n(Q_3)-2=n(R_{BL})+n(Q_3)&\equiv 4,10 \bmod 16
      \end{align*}
      This is only possible if $n(R_{BL})\equiv n(Q_3)\equiv 10\bmod 16$.
      Let $e_{BL}(R)=\{R,S\}$.
      \begin{itemize}
        \item If $\deg(S)=2$, then let $E(T)\ni e'\neq e_{BL}(R)$ be the other edge of $S$. Let the partition generated by it be $e'(P_1^{e'},P_2^{e'})$ such that $R'\subseteq P_1^{e'}$. We have
        \begin{align*}
          n(R_{BL})&=n(P_2^{e'})+n(S)+t(e') \\
          n(P_2^{e'})&=n(R_{BL})-4-t(e')\equiv 6-t(e') \bmod 16
        \end{align*}
        Either $e'$ is a 1-cut, in which case $n(P_2^{e'})\equiv 8\bmod 16$, or $e'$ is a 2-cut, giving $n(P_2^{e'})\equiv 6\bmod 16$. In any case,
        \Fref{lemma:tech} says that $e'$ is a good cut.

        \item If $\deg(S)=3$, then we can partition $P$ as in \Fref{fig:corridor_e}. Since $n(Q_5\cup Q_6)=4+n(Q_6)-2$, by \Fref{lemma:tech:a} the only case when neither $L_6(Q_5\cup Q_6,Q_4\cup R\cup P_1^e)$ nor $L_7(Q_6,Q_4\cup Q_5\cup R\cup P_1^e)$ is a good cut of $P$ is when $n(Q_6)\equiv 4,10\bmod 16$. Also,
        $$10\equiv n(Q_4\cup Q_5\cup Q_6)=n(Q_4)+n(Q_5)+n(Q_6)-2-2\equiv n(Q_4)+n(Q_6)\bmod 16.$$
        \begin{itemize}
          \item If $n(Q_6)\equiv 10\bmod 16$, then $n(Q_4)\equiv 0\bmod 16$, hence
          \begin{align*}
            n(Q_4\cup Q_5\cup Q_{11}\cup Q_{12})=n(Q_4\cup Q_5)+4-4=n(Q_4)+n(Q_5)-2\equiv 2\bmod 16,
          \end{align*}
          showing that $L_8(Q_4\cup Q_5\cup Q_{11}\cup Q_{12},Q_6\cup Q_{13}\cup Q_2\cup Q_3\cup P_1^e)$ is a good cut.
          \item If $n(Q_6)\equiv 4\bmod 16$,
          \begin{align*}
            n(Q_6\cup Q_{13}\cup Q_2\cup Q_3)&=n(Q_6)+n(Q_{13}\cup Q_2)+n(Q_3)-2-2\equiv\\
            &\equiv n(Q_6)+10\equiv 14\mod 16,
          \end{align*}
          therefore $L_9(Q_6\cup Q_{13}\cup Q_2\cup Q_3,Q_4\cup Q_5\cup Q_{11}\cup Q_{12}\cup P_1^e)$ is a good cut.
        \end{itemize}
        \begin{figure}[H]
        	\centering
        	\includefigure{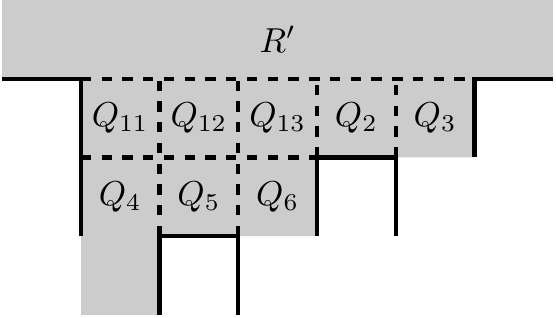}
        	\hspace{24pt}
        	\includefigure{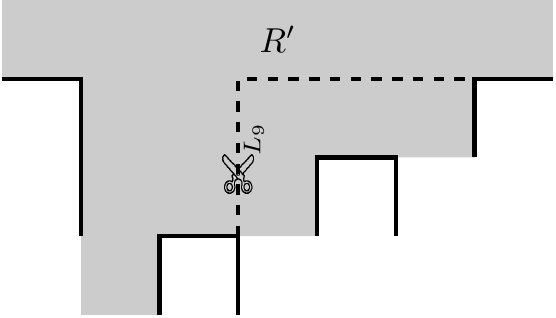}
        	\caption{$\deg(Q)\ge 3$ and $Q_{BL}\neq\emptyset$}\label{fig:corridor_e}
        \end{figure}
        In all of the above subcases we found a good cut.
      \end{itemize}
    \end{itemize}

	\end{enumerate}

\case{There are no corridors in \texorpdfstring{$T$}{T}, but there is a pocket}\label{case:pockets}

Let $S$ be a (horizontal) pocket. Also, let $R$ be the neighbor of $S$ in $T$. If $\deg(R)=2$, then \Fref{claim:deg2doublecut} provides a good cut-system of $P$. However, if $\deg(R)\ge 3$, we have two cases.

\subcase{If \texorpdfstring{$R$}{R} is adjacent to at least two pockets} 
Let $U$ be the union of $R$ and its adjacent pockets, and let $T_U$ be its \textbf{vertical} cut tree. It contains at least $4$ reflex vertices, therefore $V(T_U)\ge 3$.
		\begin{itemize}
			\item If $V(T_U)=3$, then $|E(T_U)|=2$. Thus $t(e)=0$ for any $e\in E(T_U)$, and \Fref{claim:deg2doublecut} gives a good cut-system $\mathcal{L}$ of $U$ such that all 4 vertices of $R$ are contained in $\ker\mathcal{L}$.
			\item If $V(T_U)\ge 4$, then \Fref{claim:deg2path} gives a good cut-system $\mathcal{L}$ of $U$ such that all 4 vertices of $R$ are contained in $\ker\mathcal{L}$.
		\end{itemize}
		Since there are no corridors in $P$, we have $$P=\Big(\big((U\cup R_{BL})\cup R_{TL}\big)\cup R_{BR}\Big)\cup R_{TR}.$$ By applying \Fref{lemma:extendconsecutives} repeatedly, the good cut-system $\mathcal{L}$ can be extended to a good cut-system of $P$.

\subcase{If \texorpdfstring{$S$}{S} is the only pocket adjacent to \texorpdfstring{$R$}{R}}
We may assume without loss of generality that $S$ intersect the top side of $R$. Again, define $U$ as the union of $R$ and its adjacent pockets.

		\begin{itemize}
			\item If $R_{TL}\neq\emptyset$, let $V=U\dotcup R_{TL}$. The cut-system $\{L_1,L_2,L_3\}$ in \Fref{fig:pockets1} is a good cut-system of $V$, and all 4 vertices of $R$ are contained in $\ker \{L_1,L_2,L_3\}$. By applying \Fref{lemma:extendconsecutives} repeatedly, we get a good cut-system of $P$. \figcap{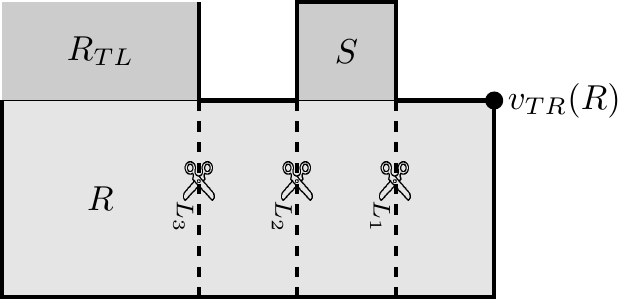}{$\{L_1,L_2,L_3\}$ is a good cut-system of $V=R\cup R_{TL}\cup S$}
			\item If $R_{TR}\neq\emptyset$, the case can be solved analogously to the previous case.
			\item Otherwise $R_{BL}\neq\emptyset$ and $R_{BR}\neq\emptyset$. Let $L_1(U_1^1,U_2^1)$ and $L_2(U_1^2,U_2^2)$ be the vertical cuts (from right to left) defined by the two reflex vertices of $U$, such that $v_{BR}(R)\in U_1^1\subset U_1^2$. Let $V=R_{BL}\dotcup U$. As before, $L_1$ and $L_2$ can be extended to cuts of $V$, say $L_1'(U_1^1,V_2^1)$, $L_2'(U_1^2,V_2^2)$. We claim that together with $e_{BL}(R)(U,V_2^3)$, they form a good cut-system $\mathcal{L}$ of $V$. This is obvious, as $\{n(U_1^1),n(U_1^2),n(U)\}=\{4,6,8\}$.
			Since $v_{BR}(R)\in \ker\mathcal{L}$, $P$ also has a good cut-system by \Fref{lemma:extendconsecutives}.
		\end{itemize}

\case{\texorpdfstring{$T$}{T} is not a path and it does not contain either corridors or pockets}\label{case:twisted}

By the assumptions of this case, any two adjacent rectangles are adjacent at one of their vertices, so the maximum degree in $T$ is 3 or 4. We distinguish between several sub-cases.
\vbox{
	\begin{description}
	\itemsep0em 
	\setlength{\baselineskip}{14pt}
	\item [\Fref{case:toporbotcontained}.] There exists a rectangle of degree \texorpdfstring{$\ge 3$}{at least 3} such that its top or bottom side is entirely contained by one of its neighboring rectangle;
	\item [\Fref{case:last}.] Every rectangle of degree \texorpdfstring{$\ge 3$}{at least 3} is such that its top and bottom side are not entirely contained by any of their neighboring rectangle;
	\begin{description}
		\itemsep0em 
		\item [\Fref{case:last1}.] There exist at least two rectangles of degree \texorpdfstring{$\ge 3$}{at least 3};
		\item [\Fref{case:last2}.] There is exactly one rectangle of degree \texorpdfstring{$\ge 3$}{at least 3}.
	\end{description}
\end{description}
}

\subcase{There exists a rectangle of degree \texorpdfstring{$\ge 3$}{at least 3} such that its top or bottom side is entirely contained by one of its neighboring rectangle}\label{case:toporbotcontained}
Let $R$ be a rectangle and $R'$ its neighbor, such that the top or bottom side of $R$ is $\subset \partial R'$. Moreover, choose $R$ such that if we partition $P$ by cutting $e=\{R,R'\}$, the part containing $R$ is minimal (in the set theoretic sense).

\medskip

Without loss of generality, the top side of $R$ is contained entirely by a neighboring rectangle $R'$ and $R_{TL}=\emptyset$.
\figcap{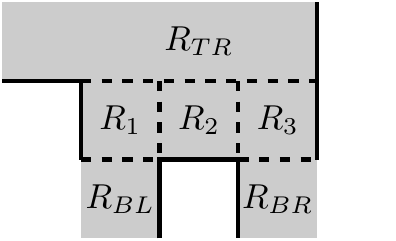}{}
This is pictured in \Fref{fig:topside}, where $R=R_1\cup R_2\cup R_3$. We can cut off $R_{BL}$, $R_{BL}\cup R_1$, and $R_{BL}\cup R_1\cup R_2$, whose number of vertices are respectively
\begin{enumerate}[\qquad\bfseries (1)]
  \setlength{\baselineskip}{16pt}
  \item $n(R_{BL})$
  \item $n(R_{BL}\cup R_1)=n(R_{BL})+n(R_1)+(t(e_{BL}(R)))-2)=n(R_{BL})+t(e_{BL}(R))+2$
  \item $n(R_{BL}\cup R_1\cup R_2)=n(R_{BL})+n(R_1\cup R_2)+t(e_{BL}(R)))=n(R_{BL})+t(e_{BL}(R))+4$
\end{enumerate}
If $t(e_{BL}(R))=0$, then one of the 3 cuts is a good cut by \Fref{lemma:tech:a}.

\medskip

Otherwise $t(e_{BL}(R))=-2$, thus either one of the 3 cuts is a good cut or $n(R_{BL})\equiv 4,10\bmod 16$. Let $S$ be the rectangle for which $e_{BL}(R)=\{R,S\}$. Since $e_{BL}(R)$ is a 1-cut containing the top side of $S$, we cannot have $\deg(S)=3$, as it contradicts the choice of $R$. We distinguish between two cases.
\subsubcase{$\deg(S)=1$}
	Let $U=R'\cup R\cup R_{BL}\cup R_{BR}$, which is depicted on \Fref{fig:topside3cuts:before}. It is easy to see that $L_1(Q_1,U_2^1)$, $L_2(Q_1\cup Q_2,U_2^2)$, and $L_3(Q_1\cup Q_2\cup Q_3,U_2^3)$ in \Fref{fig:topside3cuts:after} is a good cut-system of $U$.
	\begin{figure}[H]
		\centering
		\begin{subfigure}{.40\textwidth}
			\centering
			\includefigure{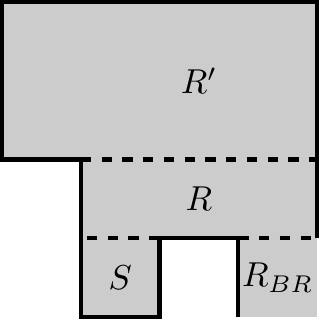}
			\caption{}
			\label{fig:topside3cuts:before}
		\end{subfigure}%
		\begin{subfigure}{.40\textwidth}
			\centering
			\includefigure{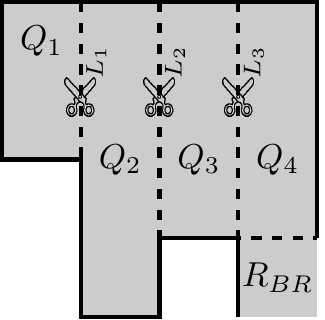}
			\caption{}
			\label{fig:topside3cuts:after}
		\end{subfigure}
		\caption{}
		\label{fig:topside3cuts}
	\end{figure}
	 As all 4 vertices of $S$ are contained in $\ker\{L_1,L_2,L_3\}$, we can extend this good cut-system to $P$ by reattaching $S_{TL}$, $S_{BL}$, $S_{TR}$ (if non-empty) via \Fref{lemma:extendconsecutives}. Therefore $P$ has a good cut.
	
\subsubcase{$\deg(S)=2$}
	Let $f$ be the edge of $S$ which is different from $e_{BL}(R)=e_{TL}(S)$. Let the partition generated by it be $f(P_1^f,P_2^f)$, where $S\subseteq P_2^f$. We have $n(P_1^f)=n(R_{BL})-n(S)-t(f)$.
	\begin{itemize}
		\item If $t(f)=-2$, then $n(P_1^f)\equiv 2,8\bmod 16$, so $f$ is a good cut by \Fref{lemma:tech:a}.
		\item If $t(f)=0$, then $n(P_1^f)\equiv 0,6\bmod 16$, so $f$ is a good cut by \Fref{lemma:tech:b}.
	\end{itemize}

\subcase{Every rectangle of degree \texorpdfstring{$\ge 3$}{at least 3} is such that its top and bottom side are not entirely contained by any of their neighboring rectangle}\label{case:last}

Let $R$ be a rectangle of degree $\ge 3$ and $e=\{R,S\}$ be one of its edges. Let the partition generated by $e$ be $e(P_1^e,P_2^e)$, where $R\subset P_1^e$ and $S\subseteq P_2^e$. If $e$ is a 1-cut then by the assumptions of this case $\deg(S)\le 2$. 
\begin{itemize}
	\item If $\deg(S)=1$ and $t(e)=-2$, then $n(P_2^e)+t(e)=2$.
	\item If $\deg(S)=1$ and $t(e)=0$, then $n(P_2^e)+t(e)=4$.
	\item If $\deg(S)=2$ and one of the edges of $S$ is a 0-cut, then $P$ has a good cut by \Fref{claim:deg2doublecut}.
	\item If $\deg(S)=2$ and both edges of $S$, $e$ and (say) $f$ are 1-cuts:
	Let the partition generated by $f$ be $P=P_1^f\dotcup P_2^f$, such that $S\in P_1^f$. Then
	$n(P_2^e)=n(P_2^f)+n(S)+t(f)=n(P_2^f)+2$. Either one of $e$ and $f$ is a good cut, or by \Fref{lemma:tech:a} we have $n(P_2^f)\equiv 4,10\mod 16$. In other words,
	$n(P_2^e)+t(e)\equiv 4,10\mod 16$. Similarly, $n(P_1^f)=n(P_1^e)+4-2=n(P_1^e)+2$, so $n(P_1^e)\equiv 4,10\mod 16$.
	\item If $\deg(S)\ge 3$, then $t(e)=0$. Either $e$ is a good cut, or by \Fref{lemma:tech:b} we have $n(P_2^e)+t(e)\equiv 4,10\mod 16$. \Fref{lemma:tech:b} also implies $n(P_1^e)\equiv 4,10\mod 16$.
\end{itemize}
From now on, we assume that none of the edges of the neighbors of a degree $\ge 3$ rectangle represent a good cut, so in particular we have $$n(P_2^e)+t(e)\equiv 2,4,\text{ or }10\mod 16.$$

In addition to the simple analysis we have just conducted, we deduce an easy claim to be used in the following subcases.

\begin{claim}\label{claim:opposite}
	Let $R\in V(T)$ be of degree $\ge 3$ and suppose both $R_{BR}\neq\emptyset$ and $R_{TR}\neq\emptyset$. Then $P$ has two admissible cuts $L_1$ and $L_2$ such that they form a good cut-system or
	$$
	\begin{array}{rl}
		\textbf{(i)}&\text{one of the parts generated by $L_1$ has size }\Big(n(R_{BR})+t(e_{BR}(R))\Big)+\Big(n(R_{TR})+t(e_{TR}(R))\Big)+2,\makeatletter\edef\@currentlabel{\arabic{theorem}(i)}\makeatother\label{claim:opposite:2} \\
		\mbox{}&\textbf{\quad and }\\
		\textbf{(ii)}&\text{one of the parts generated by $L_2$ has size }
		\Big(n(R_{BR})+t(e_{BR}(R))\Big)+\Big(n(R_{TR})+t(e_{TR}(R))\Big)+4.\makeatletter\edef\@currentlabel{\arabic{theorem}(ii)}\makeatother\label{claim:opposite:4}
	\end{array}
	$$
\end{claim}
\begin{proof}
	Let $U=R\cup R_{BL}\cup R_{BR}$. Let $L_1(U_1^1,U_2^1)$ and $L_2(U_1^2,U_2^2)$ be the vertical cuts of $U$ 
	defined by the two reflex vertices of $U$ that are on the boundary of $R$, such that $v_{BR}(R)\in U_1^1\subset U_1^2$. By \Fref{lemma:extend}, $L_1$ and $L_2$ can be extended to cuts of $V=R\cup R_{BL}\cup R_{BR}\cup R_{TR}$, say $L_1'(V_1^1,U_2^1)$, $L_2'(V_1^2,U_2^2)$. If one of $L_1'$ or $L_2'$ is a 2-cut, then similarly to \Fref{claim:deg2doublecut} one can verify they form a good cut-system of $V$ which we can extend to $P$. Otherwise
	\begin{align*}
		n(V_1^1)&=n(R\cap U_1^1)+n(R_{BR})+n(R_{TR})+(t(e_{BR}(R))-2)+t(e_{TR}(R))=\\
		&=\Big(n(R_{BR})+t(e_{BR}(R))\Big)+\Big(n(R_{TR})+t(e_{TR}(R))\Big)+2,\\
		n(V_1^2)&=n(R\cap U_1^2)+n(R_{BR})+n(R_{TR})+t(e_{BR}(R))+t(e_{TR}(R))=\\
		&=\Big(n(R_{BR})+t(e_{BR}(R))\Big)+\Big(n(R_{TR})+t(e_{TR}(R))\Big)+4.
	\end{align*}
	Lastly, we extend $L_1'$ and $L_2'$ to $P$ by reattaching $R_{TL}$ using \Fref{lemma:extend}. This step does not affect the parts $V_1^1$ and $V_1^2$, so we are done.
\end{proof}

\subsubcase{There exist at least two rectangles of degree \texorpdfstring{$\ge 3$}{at least 3}}\label{case:last1}
In the subgraph $T'$ of $T$ which is the union of all paths of $T$ which connect two degree $\ge 3$ rectangles, let $R$ be a leaf and $e=\{R,S\}$ its edge in the subgraph. As defined in the beginning of \Fref{case:last}, the set of incident edges of $R$ (in $T$) is $\{e_i\ |\ 1\le i\le \deg(R)\}$, and without loss of generality we may suppose that $e=e_{\deg(R)}$. The analysis also implies that for all $1\le i\le \deg(R)-1$ we have $n(P_2^{e_i})+t(e_i)=2,4$. 

\medskip

By the assumptions of this case $\deg(S)\ge 2$, therefore $n(P_1^e)\equiv 4\text{ or }10\mod 16$. If $\deg(S)\ge 3$ let $Q=S$. Otherwise $\deg(S)=2$ and let $Q$ be the second neighbor of $R$ in $T'$. The degree of $Q$ cannot be 1 by its choice. If $\deg(Q)=2$, then we find a good cut using \Fref{claim:deg2path}. In any case, we may suppose from now on that $\deg(Q)\ge 3$.

\medskip

Let $\{f_i\ |\ 1\le i\le \deg(Q)\}$ be the set of incident edges of $Q$ such that they generate the partitions $P=P_1^{f_i}\dotcup P_2^{f_i}$ where $R\subset P_1^{f_i}$ and $Q\subset P_2^{f_1}$. We have
\begin{align*}
	n(P_1^{e})&=n(R)+\sum_{i=1}^{\deg(R)-1}\Big(n(P_2^{e_i})+t(e_i)\Big) \in 4+\{2,4\}+\{2,4\}+\{0,2,4\}=\{8,10,12,14,16\},
\end{align*}
so the only possibility is $n(P_1^e)=10$.
\begin{itemize}
	\item If $\deg(S)\ge 3$, $e$ is a 2-cut (by the assumption of \Fref{case:last}), so by \Fref{lemma:tech:c} either $e$ is a good cut or $n(P)\equiv 14\bmod 16$. Since $Q=S$ and $e=f_1$, we have 
	\begin{align*}
		n(P_1^{f_1})+t(f_1)&=n(P_1^e)+t(e)= 10 \\
		n(P_2^{f_1})&=n(P)-n(P_1^{f_1})-t(f_1)\equiv 14-10\equiv 4\bmod 16.
	\end{align*}
	\item If $\deg(S)=2$, either $e$ is a 1-cut or we find a good cut using \Fref{claim:deg2doublecut}. Also, $f_1=\{S,Q\}$ is a 1-cut too (otherwise apply \Fref{claim:deg2doublecut}), so $$n(P_1^{f_1})+t(f_1)=n(P_1^e)+n(S)+t(e)+t(f_1)=10.$$ By \Fref{lemma:tech:c}, either $f_1$ is a good cut (as $n(P_1^{f_1})=12$) or $n(P)\equiv 14\bmod 16$. Thus
	$$n(P_2^{f_1})=n(P)-n(P_1^{f_1})-t(f_1)\equiv 14-12+2\equiv 4\bmod 16$$
\end{itemize}
We have 
\begin{align*}
	n(P)&=n(Q)+\Big(n(P_1^{f_1})+t(f_1)\Big)+\sum_{i=2}^{\deg(Q)}\Big(n(P_2^{f_i})+t(f_i)\Big)\in \\ &\in 14+\{2,4,10\}+\{2,4,10\}+\{0,2,4,10\} \mod 16.
\end{align*}
The only way we can get $14\bmod 16$ on the right hand side is when $\deg(Q)=4$ and out of $$n(Q_{BL}),n(Q_{BR}),n(Q_{TL}),n(Q_{TR})\mod 16$$ one is $2$, another is $4$, and two are $10 \bmod 16$.

\medskip

The last step in this case is to apply \Fref{claim:opposite} to $Q$. If it does not give a good cut-system then it gives an admissible cut where one of the parts has size congruent to $2+10+2=14$ or $2+4+2=8$ modulo~16, therefore we find a good cut anyway.

\subsubcase{There is exactly one rectangle of degree \texorpdfstring{$\ge 3$}{at least 3}}\label{case:last2}
Let $R$ be the rectangle of degree $\ge 3$ in $T$, and let $\{e_i \ |\ 1\le i\le \deg(R)\}$ the edges of $R$, which generate the partitions $P=P_1^{e_i}\dotcup P_2^{e_i}$ where $R\subset P_1^{e_i}$. Then $P_2^{e_i}$ is path for all $i$. If either \Fref{claim:deg2doublecut} or \Fref{claim:deg2path} can be applied, $P$ has a good cut. The remaining possibilities can be categorized into 3 types:
$$
\begin{array}{lllll}
	\text{\bfseries Type 1: } & t(e_i)=-2, & n(P_2^{e_i})=4, & \text{ and } & n(P_2^{e_i})+t(e_i)=2;\\
	\text{\bfseries Type 2: } & t(e_i)=-2, & n(P_2^{e_i})=4+4-2=6, & \text{ and } & n(P_2^{e_i})+t(e_i)=4;\\
	\text{\bfseries Type 3: } & t(e_i)=0, & n(P_2^{e_i})=4, & \text{ and } & n(P_2^{e_i})+t(e_i)=4.
\end{array}
$$
Without loss of generality $R_{BR}\neq\emptyset$ and $R_{TR}\neq\emptyset$. We will now use \Fref{claim:opposite}. If it gives a good cut-system, we are done. Otherwise
\begin{itemize}
	\item If exactly one of $e_{BR}(R)$ and $e_{TR}(R)$ is of {\bfseries type~1}, apply \Fref{claim:opposite:2}: it gives an admissible cut which cuts off a subpolyomino of size $2+4+2=8$, so $P$ has a good cut.
	\item If both $e_{BR}(R)$ and $e_{TR}(R)$ are of {\bfseries type~1}, apply \Fref{claim:opposite:4}: it gives an admissible cut which cuts off a subpolyomino of size $2+2+4=8$, so $P$ has a good cut.
	\item If none of $e_{BR}(R)$ and $e_{TR}(R)$ are of {\bfseries type~1} and $n(P)\not\equiv 14 \pmod{16}$, apply  \Fref{claim:opposite:2}: it gives an admissible cut which cuts off a subpolyomino of size $4+4+4=12$, which is a good cut by \Fref{lemma:tech:c}.
\end{itemize}
Now we only need to deal with the case where $n(P)=14$ and neither $e_{BR}(R)$ nor $e_{TR}(R)$ is of {\bfseries type~1}.
%The proof \Fref{thm:mobile} is completed in the following paragraphs.

\medskip

If $R$ still has two edges of {\bfseries type~1}, again \Fref{claim:opposite:2} gives a good cut of $P$. If $R$ has at most one edge of {\bfseries type~1}, we have
$$14=n(P)=n(R)+\sum_{i=1}^{\deg(R)}\Big(n(P_2^{e_i})+t(e_i)\Big)\in 4+\{2,4\}+\{4\}+\{4\}+\{0,4\}=\{14,16,18,20\},$$
and the only way we can get $14$ on the right hand is when $\deg(R)=3$ and both $e_{BR}(R)$ and $e_{TR}(R)$ are of {\bfseries type~2~or~3} while the third incident edge of $R$ is of {\bfseries type~1}. We may  assume without loss of generality that the cut represented by $e_{TR}(R)$ is longer than the cut represented by $e_{BR}(R)$. 
\begin{itemize}
	\item If $P$ is vertically convex, its vertical cut tree is a path, so it has a good cut as deduced in \Fref{case:Tisapath}.
	\item If $P$ is not vertically convex, but $e_{TR}(R)$ is a {\bfseries type~2} edge of $R$, such that the only horizontal cut of $R_{TR}$ is shorter than the cut represented by $e_{TR}(R)$, then $P'=P-R_{BR}$ is vertically convex, and has $(10-4)/2=3$ reflex vertices. By \Fref{claim:deg2doublecut} or \Fref{claim:deg2path}, $P'$ has a good cut-system such that its kernel contains $v_{BR}(R)$, since its $x$ coordinate is maximal in $P'$. \Fref{lemma:extendconsecutives} states that $P$ also has a good cut-system.
	\item Otherwise we find that the top right part of $P$ looks like to one of the cases in \Fref{fig:44}. It is easy to see that in all three pictures $L$ is an admissible cut which generates two subpolyominoes of 8 vertices.
	\begin{figure}[H]
		\centering
		\begin{subfigure}{.33\textwidth}
			\centering
			\includefigure{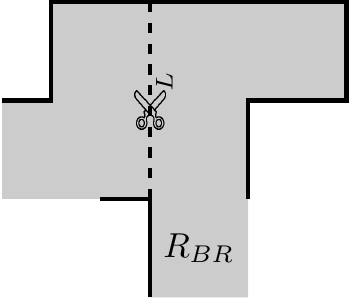}
			\caption{}\label{fig:44:a}
		\end{subfigure}%
		\begin{subfigure}{.33\textwidth}
			\centering
			\includefigure{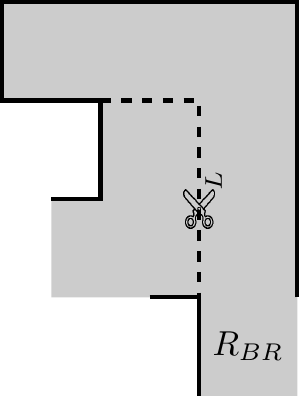}
			\caption{}\label{fig:44:b}
		\end{subfigure}%
		\begin{subfigure}{.33\textwidth}
			\centering
			\includefigure{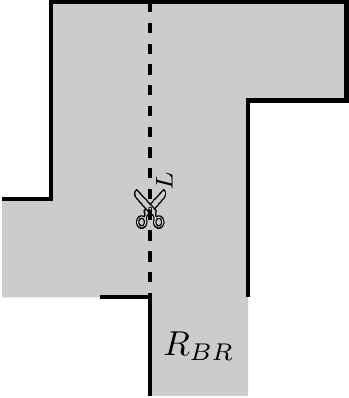}
			\caption{}\label{fig:44:c}
		\end{subfigure}%
		
		\caption{}\label{fig:44}
	\end{figure}
\end{itemize}

The proof of \Fref{thm:mobile} is complete. To complement the formal proof, we now demonstrate the algorithm on \Fref{fig:partition}.

\figcap{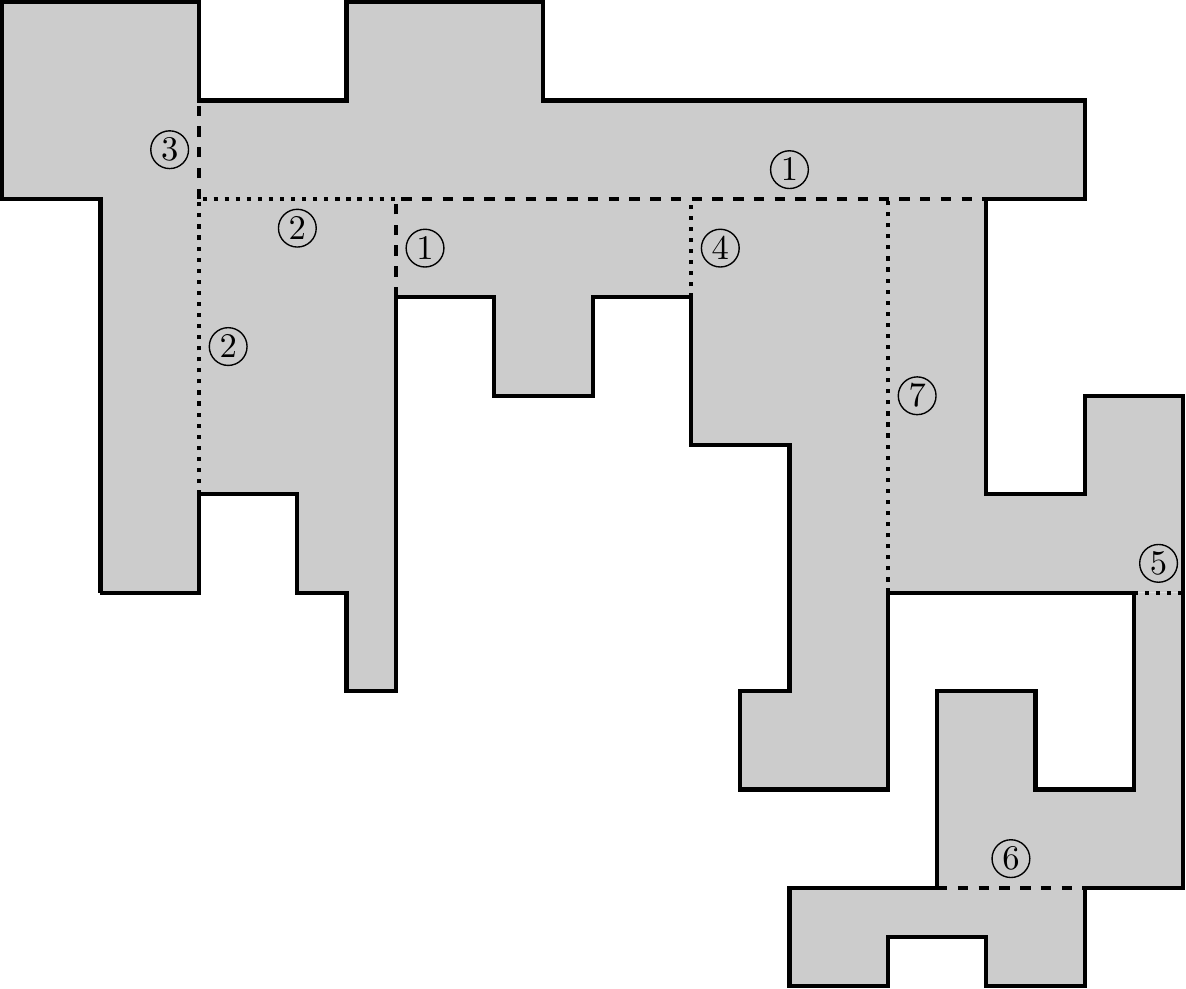}{The output of the algorithm on a polyomino $n=52$ vertices.}

First, we resolve a corridor via the $L$-cut \circled{1}, which creates two pieces of 20 and 34($\equiv 2\bmod 16$) vertices. As a result of this cut, a new corridor emerges in the 20-vertex piece, so we cut the polyomino at \circled{2}, cutting off a piece of 8 vertices. The other piece of 14 vertices containing two pockets is further divided by \circled{3} into a piece of 6 and 8 vertices. Another pocket is dealt with by cut \circled{4}, which divides the polyomino into an 8- and a 28-vertex piece. To the larger piece, \Fref{case:last} applies, and we find cut \circled{5}, which produces a 16- and a 14-vertex piece. The 16-vertex piece is cut into two 8-vertex pieces by \circled{6}. Lastly, \Fref{fig:44:c} of \Fref{case:last2} applies to the 14-vertex piece, so cut \circled{7} divides it into two 8-vertex pieces.

\medskip

We got lucky with cuts \circled{2} and \circled{6} in the sense that they both satisfy the inequality in (\ref{eq:n1n2}) strictly. Hence, only 8 pieces are needed to partition the polyomino on \Fref{fig:partition} into polyominoes of $\le 8$-vertex pieces, instead of the upper bound of $(3\cdot 52+4)/16=10$.

\section{Conclusion}\label{sec:conclusion}
We have not dealt with algorithmic aspects yet in this paper, but the proof given in the previous section can easily be turned into an $O(n^2)$ algorithm which partitions $P$ into at most $\ff{n}$ simple polyominoes. The running time can be improved to $O(n)$ by using linear-time triangulation \citep{chazelle} to construct the horizontal cut tree of $P$ (such that the edge list of a vertex is ordered by the $x$ coordinates of the corresponding cuts) and a list of pockets, corridors, and rectangles of \Fref{case:toporbotcontained}, all of which can be maintained in $O(1)$ for the partitions after finding and performing a cut in $O(1)$.

\medskip

\Fref{thm:mobile} fills a gap between two already established (sharp) results: in \citep{Gy} it is proved that polyominoes can be partitioned into at most $\lfloor \frac{n}{4}\rfloor$ polyominoes of at most 6 vertices, and in \citep{GyH} it is proved that any polyomino in general position (a polyomino without 2-cuts) can be partitioned into $\lfloor \frac{n}{6}\rfloor$ polyominoes of at most $10$ vertices. However, we do not know of a sharp theorem about partitioning polyominoes into polyominoes of at most 12 vertices.

\medskip 

Furthermore, for $k\ge 4$, not much is known about partitioning (not necessarily simply connected) orthogonal polygons into polyominoes of at most $2k$ vertices. According to the ``metatheorem,'' the first step in this direction would be proving that an orthogonal polygon of $n$ vertices with $h$ holes can be partitioned into $\lfloor \frac{3n+4h+4}{16}\rfloor$ polyominoes of at most 8 vertices. This would generalize the corresponding art gallery result in \citep[Thm.~5.]{GyH}.

\section*{Acknowledgment}
The authors would like to thank Joseph O'Rourke and an anonymous referee for their helpful comments and remarks.

\bibliography{references}

\end{document}